\documentclass[12pt]{amsart}
\usepackage{amsmath,amssymb,amsbsy,amsfonts,amsthm,latexsym,amsopn,amstext,
                                                    amsxtra,euscript,amscd}
\usepackage[english]{babel}
\begin{document}

\newtheorem{thm}{Theorem}
\newtheorem{lem}[thm]{Lemma}
\newtheorem{claim}[thm]{Claim}
\newtheorem{cor}[thm]{Corollary}
\newtheorem{prop}[thm]{Proposition} 
\newtheorem{definition}{Definition}
\newtheorem{question}[thm]{Open Question}
\newtheorem{conj}[thm]{Conjecture}
\newtheorem{prob}{Problem}
\def\vol {{\mathrm{vol\,}}}
\def\squareforqed{\hbox{\rlap{$\sqcap$}$\sqcup$}}
\def\qed{\ifmmode\squareforqed\else{\unskip\nobreak\hfil
\penalty50\hskip1em\null\nobreak\hfil\squareforqed
\parfillskip=0pt\finalhyphendemerits=0\endgraf}\fi}

\def\cA{{\mathcal A}}
\def\cB{{\mathcal B}}
\def\cC{{\mathcal C}}
\def\cD{{\mathcal D}}
\def\cE{{\mathcal E}}
\def\cF{{\mathcal F}}
\def\cG{{\mathcal G}}
\def\cH{{\mathcal H}}
\def\cI{{\mathcal I}}
\def\cJ{{\mathcal J}}
\def\cK{{\mathcal K}}
\def\cL{{\mathcal L}}
\def\cM{{\mathcal M}}
\def\cN{{\mathcal N}}
\def\cO{{\mathcal O}}
\def\cP{{\mathcal P}}
\def\cQ{{\mathcal Q}}
\def\cR{{\mathcal R}}
\def\cS{{\mathcal S}}
\def\cT{{\mathcal T}}
\def\cU{{\mathcal U}}
\def\cV{{\mathcal V}}
\def\cW{{\mathcal W}}
\def\cX{{\mathcal X}}
\def\cY{{\mathcal Y}}
\def\cZ{{\mathcal Z}}

\def\NmQR{N(m;Q,R)}
\def\VmQR{\cV(m;Q,R)}

\def\Xm{\cX_m}

\def \C {{\mathbb C}}
\def \F {{\mathbb F}}
\def \L {{\mathbb L}}
\def \K {{\mathbb K}}
\def \Q {{\mathbb Q}}
\def \R {{\mathbb R}}
\def \Z {{\mathbb Z}}
\def \fS{\mathfrak S}

\def\\{\cr}
\def\({\left(}
\def\){\right)}
\def\fl#1{\left\lfloor#1\right\rfloor}
\def\rf#1{\left\lceil#1\right\rceil}

\def \bFp {\overline \F_p}

\newcommand{\pfrac}[2]{{\left(\frac{#1}{#2}\right)}}

\def \Prob{{\mathrm {}}}
\def\e{\mathbf{e}}
\def\ep{{\mathbf{\,e}}_p}
\def\epp{{\mathbf{\,e}}_{p^2}}
\def\em{{\mathbf{\,e}}_m}

\def\Res{\mathrm{Res}}

\def\vec#1{\mathbf{#1}}
\def\flp#1{{\left\langle#1\right\rangle}_p}

\def\mand{\qquad\mbox{and}\qquad}

\newcommand{\comm}[1]{\marginpar{%
\vskip-\baselineskip 
\raggedright\footnotesize
\itshape\hrule\smallskip#1\par\smallskip\hrule}}

\title{Subgroups Generated by Rational Functions in Finite Fields}

\author{Domingo G\'omez-P\'erez} 
\address{Department of Mathematics, University of Cantabria, Santander 39005, Spain}
\email{domingo.gomez@unican.es}

\author{Igor E. Shparlinski} 
\address{Department of Pure Mathematics, University of New South Wales, 
Sydney, NSW 2052, Australia}
\email{igor.shparlinski@unsw.edu.au}

\date{\today}

\begin{abstract} For a large prime $p$, 
a rational function $\psi \in\F_p(X)$ over the finite field 
$\F_p$ of $p$ elements, 
and  integers $u$ and $H\ge 1$, we obtain 
a lower bound on the number  consecutive values $\psi(x)$, 
$x = u+1, \ldots, u+H$ that belong to a given multiplicative 
subgroup of $\F_p^*$. 
\end{abstract}



\maketitle

\section{Introduction}

For a prime $p$, let $\F_p$ denote the finite field with $p$ elements,
which we always assume to be represented by the set $\{0, \ldots, p-1\}$.

Given a rational function 
$$
\psi(X) = \frac{f(X)}{g(X)}\in \F_p(X)
$$ 
where  $f, g\in\F_p[X]$ are relatively prime polynomials, 
and an `interesting' set $\cS \subseteq \F_p$, it is natural to ask 
how the value set 
$$
\psi(\cS) = \{\psi(x)~:~x \in \cS, \ g(x) \ne 0\}
$$
is distributed. For instance, given another `interesting' set $\cT$, 
our goal is to obtain nontrivial bounds on the size of the 
intersection
$$
N_\psi(\cS,\cT) = \#\(\psi(\cS) \cap \cT\).
$$
In particular, we are interested in the cases when
$N_\psi(\cS,\cT)$ achieves  the trivial upper bound 
$$
N_\psi(\cS,\cT) \le \min\{\#\cS, \# \cT\}.
$$

Typical examples of such sets $\cS$ and $\cT$ are given by 
intervals $\cI$ of consecutive 
integers and multiplicative subgroups $\cG$ of $\F_p^*$. 
For large intervals and subgroups, a standard application 
of bounds of exponential and multiplicative character sums
leads to asymptotic formulas for the relevant values of $N_\psi(\cS,\cT)$,
see~\cite{CCGHSZ,CGOS,Shp-GG}. Thus only the case of 
small intervals and groups is of interest.

For a polynomial $f \in\F_p[X]$ and two intervals $\cI = \{u+1, \ldots, u+H\}$ 
and $\cJ = \{v+1, \ldots, v+H\}$ of $H$ consecutive 
integers, various bounds on 
the cardinality of the intersection  $f(\cI) \cap \cJ$ are 
given in~\cite{CCGHSZ,CGOS}. To present some of these results, 
for positive integers $d$, $k$ and $H$, we denote by
$J_{d,k}(H)$ the number of solutions to the system of
equations
$$
x_1^{\nu}+\ldots +x_k^{\nu}=x_{k+1}^{\nu}+\ldots +x_{2k}^{\nu},
\qquad \nu = 1, \ldots, d,
$$
in positive integers $x_1,\ldots,x_{2k}\le H$.
Then by~\cite[Theorem~1]{CGOS}, for any $f \in \F_p[X]$ 
of degree $d \ge 2$  
and two intervals $\cI$ 
and $\cJ$ of  $H<p$ consecutive 
integers, we have
$$
N_f(\cI,\cJ) \le  H (H/p)^{1/2\kappa(d)+o(1)} + H^{1- (d-1)/2\kappa(d)+o(1)},
$$
as $H\to \infty$, 
where $\kappa(d)$ is the smallest integer $\kappa$ such
that for $k \ge \kappa$ there exists a constant
$C(d,k)$ depending only on $k$ and $d$ and such that 
$$
J_{d,k}(H) \le C(d,k) H^{2k - d(d+1)/2+o(1)}
$$
holds as   $H\to \infty$, see also~\cite{CCGHSZ} for some improvements 
and results for  related problems. 
In~\cite{CCGHSZ,CGOS} the bounds of
Wooley~\cite{Wool1,Wool2} 
are used that give the presently best known estimates on $\kappa(d)$ (at least for 
a large $d$), see also~\cite{Wool3} for further progress in estimating $\kappa(d)$.

It is easy to see that the argument of the proof of~\cite[Theorem~1]{CGOS} allows to 
consider intervals of $\cI$ and $\cJ$ of different lengths as well and for 
intervals 
$$
\cI = \{u+1, \ldots, u+H\} 
\mand \cJ = \{v+1, \ldots, v+K\}$$ 
with $1 \le H, K < p$ it leads to the bound
$$
N_f(\cI,\cJ) \le  H^{1+o(1)} \((K/p)^{1/2\kappa(d)} + (K/H^d)^{1/2\kappa(d)}\),
$$
see also a more general 
result of Kerr~\cite[Theorem~3.1]{Ker}  that applies to multivariate 
polynomials and to congruences modulo a composite 
number.

Furthermore, let $K_\psi(H)$ be the smallest $K$ for which there 
are intervals $\cI = \{u+1, \ldots, u+H\}$ 
and $\cJ = \{v+1, \ldots, v+K\}$ for which $N_\psi(\cI,\cJ) = \#\cI$.
That is, $K_\psi(H)$ is the length of the shortest interval, which 
may contain $H$ consecutive values of  $\psi \in\F_p(X)$ of degree $d$. 

Defining $\kappa^*(d)$ in the same way as $\kappa(d)$, however with 
respect to the more precise bound 
$$
J_{d,k}(H) \le C(d,k) H^{2k - d(d+1)/2}
$$
(that is, without $o(1)$ in the exponent) we can easily derive that
for any polynomial $f \in \F_p[X]$ of degree $d$, 
\begin{equation}
\label{eq:KH}
K_f(H) = O(H^d).
\end{equation}
To see that the bound~\eqref{eq:KH} is optimal it is enough to take 
$f(X) = X^d$ and $u=0$. Note that the proof of~\eqref{eq:KH}
depends only on the existence of  $\kappa^*(d)$ rather than 
on its specific bounds. 
However, we recall that Wooley~\cite[Theorem~1.2]{Wool1} shows that 
for some constant $\fS(d,k)> 0$ depending 
only on $d$ and $k$ we have
$$
J_{d,k}(H) \sim \fS(d,k) H^{2k - d(d+1)/2}
$$ 
for any fixed $d\ge 3$ and $k\ge d^2 + d + 1$.
In particular, $\kappa^*(d) \le d^2 + d + 1$.

Here we concentrate on estimating $N_\psi(\cI, \cG)$ 
for an interval $\cI$  of $H$ consecutive 
integers and a multiplicative subgroup $\cG \subseteq \F_p^*$ of order $T$. 
This question has been mentioned in~\cite[Section~4]{CGOS}
as an open problem. 

We remark that for linear polynomials $f$
the result of~\cite[Corollary~34]{BGKS1}
have a natural interpretation as a lower bound on the
order of  a subgroup $\cG\subseteq \F_p^*$ for which  $N_f(\cI, \cG) = \#\cI$. 
In particular, we infer from~\cite[Corollary~34]{BGKS1}
that for any linear polynomials $f(X) = aX + b\in \F_p[X]$
and  fixed integer $\nu =1, 2, \ldots$, 
for an interval $\cI$  of  $H \le p^{1/(\nu^2-1)}$ 
consecutive integers and a subgroup $\cG$, the equality 
$N_f(\cI, \cG) = \#\cI$ 
implies $\# \cG \ge H^{\nu + o(1)}$.

We also remark that 
the results of~\cite[Section~5]{BGKS2} have a similar interpretation
for the identity $N_f(\cI, \cG) = \#\cI$ with linear polynomials, 
however apply to almost all 
primes $p$ (rather than to all primes).

Furthermore, a result of Bourgain~\cite[Theorem~2]{Bour}
gives a nontrivial bound on the intersection of an interval centered at $0$, 
that is, of the form $\cI = \{0, \pm 1, \ldots, \pm H\}$
and a co-set $a\cG$ (with $a \in \F_p^*$) 
of a multiplicative group $\cG \subseteq \F_p^*$, provided
that $H < p^{1-\varepsilon}$ and $\#\cG \ge g_0(\varepsilon)$,
for some constant  $g_0(\varepsilon)$ depending only on 
an arbitrary $\varepsilon> 0$. 

We note that several bounds on $ \#\(f(\cG) \cap \cG\)$ 
for a  multiplicative subgroup $\cG \subseteq \F_p^*$ 
are given in~\cite{Shp-GG}, but they apply only to polynomials
$f$ defined over $\Z$ and are not uniform with respect  to the 
height (that is, the size of the coefficients) of $f$. 
Thus the question of estimating $N_f(\cG, \cG)$ remains open.
On the other hand, a number of results  about points on 
curves and algebraic varieties with coordinates from small 
subgroups, in particular, in relation to the 
{\it Poonen Conjecture\/},  
have been given in~\cite{BCGH-VJMMRS,Chang1,Chang2,CKSZ,Pop1,Pop2,Vol1,Vol2}.

We recall that
the notations $U = O(V)$, $U \ll V$ and  $V \gg U$  are all
equivalent to the statement that the inequality $|U| \le c\,V$ holds
with some constant $c> 0$. Throughout the paper, any implied constants
in these symbols
may occasionally depend, where obvious, on   $d=\deg f$ and $e=\deg g$, but are absolute otherwise.

\section{Preparations}

\subsection{Absolute irreducibility of some polynomials}
As usual, we use $\bFp$ to denote the algebraic closure of $\F_p$ and 
$X, Y$ to denote indeterminate variables. We also use $\bFp(X)$, $\bFp(Y)$, 
$\bFp(X,Y)$ to denote
the corresponding fields of rational functions over $\bFp$. 

We recall that the degree of
a rational function in the variables $X,Y$ 
$$
  F(X,Y) = \frac{s(X,Y)}{t(X,Y)}\in \bFp(X,Y),\qquad \gcd(s(X,Y), t(X,Y))= 1,
$$
is $\deg F= \max \{\deg s, \deg t\}$. 

It is also known that 
if $R(X)\in \bFp(X)$ is an rational function then 
\begin{equation}
\label{eq:deg prod}
\deg (R\circ F) = \deg R \deg F,
\end{equation}
where $\circ $ denotes the composition.

We use the following result of Bodin~\cite[Theorem~5.3]{Bod} adapted
to our purposes. 

\begin{lem}
 \label{lem:Bodin}
Let $s(X,Y),t(X,Y)\in\F_p[X,Y]$  be  polynomials such that there
does not exist a   rational function $R(X)\in\bFp(X)$ with 
$\deg R > 1$ and a bivariate  rational function $G(X,Y)\in\bFp[X,Y]$ 
such  that,
$$
F(X,Y) =  \frac{s(X,Y)}{t(X,Y)} = R(G(X,Y)).
$$
The number of elements $\lambda$ such that the polynomial 
$s(X,Y)-\lambda t(X,Y)$ is reducible over $\bFp[X,Y]$ is at most
$(\deg F)^2$.
\end{lem}

We say that a rational function $f\in \bFp(X)$ is a {\it perfect
  power\/} of another 
rational function if and only if $f(X) = (g(X))^n$ for some  rational
function $g(X)\in \bFp(X)$
and integer $n\ge 2$. Because $\bFp$ is algebraic closed field, it is
trivial to see that if $f(X)$ is a perfect power, then $af(X)$ is also
a perfect power for any $a\in\bFp$. We need the following easy
technical lemma.
\begin{lem}
  \label{lem:uni_factors}
  Let $P_1(X), Q_1(X)\in\bFp[X]$, $P_2(Y),Q_2(Y)\in\bFp[Y]$
by relatively prime polynomials. Then the following bivariate
  polynomial 
  \begin{equation*}
    r P_1(X)Q_2(Y) - s Q_1(X)P_2(Y),\quad
    r,s\in\bFp^*,
  \end{equation*}
  is not divisible by any univariate polynomial.
\end{lem}

\begin{proof}
  Suppose that this polynomial was divisible by an univariate
  polynomial $d(X)$. Take $\alpha\in\bFp$ any root of the polynomial
  $d$ and substitute it getting,
  \begin{equation*}
    r P_1(\alpha)Q_2(Y) - s Q_1(\alpha)P_2(Y) =0\implies
    Q_2(Y) =\frac{s Q_1(\alpha)P_2(Y)}{ r P_1(\alpha)}. 
  \end{equation*}
  Here, we have two different possibilities:
  \begin{itemize}
  \item If $ r P_1(\alpha)=0$, then $Q_1(\alpha)=0$, and we get a
    contradiction,
  \item In other case, $\gcd(Q_2(Y),P_2(Y))\neq 1$, contradicting our
    hypothesis. 
  \end{itemize}
  This comment finishes the proof.
\end{proof}

Now, we prove the following result about
irreducibility. 

\begin{lem}
\label{lem:LinFact}
Given relatively prime polynomials $ f,g\in \bFp[X]$ and
if a rational function $f(X)/g(X)\in \bFp(X)$ of degree $D\ge 2$ 
is not a perfect power  then $f(X)g(Y) - \lambda f(Y)g(X)$ 
is reducible over $\bFp[X,Y]$ for at most $4D^2$ values 
of $\lambda \in \bFp^*$.
\end{lem}

\begin{proof}
First we describe the idea of the proof. 
Our aim is to show that the condition of Lemma~\ref{lem:Bodin} 
holds for  the polynomial $f(X)g(Y)-\lambda f(Y)g(X)$. Indeed, we 
show that if
\begin{equation}
\label{eq:linfact}
  \frac{f(X)g(Y)}{g(X)f(Y)} = R(G(X,Y)),   
  \end{equation}
with a rational function $R\in \bFp(X)$ of degree $\deg R\ge 2$ 
and a bivariate rational function $G(X,Y) \in \bFp(X,Y)$, then  there
exists another $\widetilde{R}\in \bFp(X)$ and $\widetilde{G}(X,Y) \in \bFp(X,Y)$
$$
 \frac{f(X)g(Y)}{g(X)f(Y)} = \(\widetilde R\(\widetilde G(X,Y)\)\)^m,
$$ 
for an appropiate integer $m\ge 2$. Comparing coefficients, it is easy to
arrive at the conclusion that $f(X)/g(X)$ is a perfect power.
 
Without loss of generality, we suppose $R(0)=0$. 
So, indeed we have
$$
  R(X) = a  \frac{X\prod_{i=2}^{k}  (X-r_i)}{\prod_{j=1}^{m} (X-s_j)}.
$$
Writing $G(X,Y)= G_1(X,Y)/G_2(X,Y)$ in its lowest terms  
and by hypothesis, we have that
the  fraction on the right of this inequality,
\begin{equation*}
\begin{split}
   \frac{f(X)g(Y)}{g(X)f(Y)}& = 
  a \frac{G_2(X,Y)^{N-k}}{G_2(X,Y)^{N-m}}\\
&\qquad \quad \cdot
  \frac{G_1(X,Y)\prod_{i=2}^{k} 
    (G_1(X,Y)-r_i(G_2(X,Y))}{\prod_{j=1}^{m} (G_1(X,Y)-s_jG_2(X,Y))}, 
\end{split}
\end{equation*}
where 
$$
N = \max\{k,m\}
$$
is in its lowest terms. This means that $G_1(X,Y) = P_1(X)P_2(Y)$ and 
$G_2(X,Y) = s_1^{-1}(P_1(X)P_2(Y)-Q_1(X)Q_2(Y))$, where  $P_1,P_2,Q_1, Q_2$
are  divisors of $f$ or $ g$. 
Because $\gcd(G_1(X,Y),G_2(X,Y)) = 1$, we have that 
$$\gcd (P_1(X),Q_1(X))=\gcd(P_2(Y),Q_2(Y))=1.$$

Lemma~\ref{lem:uni_factors} implies that $m= k$ as otherwise 
$G_2(X,Y)$ is divisible by an univariate polynomial.
This implies, 
$$
  \frac{f(X)g(Y)}{g(X)f(Y)} = 
  a  \frac{G_1(X,Y)\prod_{i=2}^{m}
    (G_1(X,Y)-r_iG_2(X,Y))}{\prod_{j=1}^{m} (G_1(X,Y)-s_jG_2(X,Y))}. 
$$
Now, suppose that there exists another value 
$$
s\in\{r_2,\ldots,
r_m,s_2,\ldots, s_m\}, \qquad s\neq 0,s_1.
$$ 
Then, the following polynomial
$$
G_1(X,Y)-sG_2(X,Y) = (1-s s_1^{-1})P_1(X)P_2(Y) + s_1^{-1}Q_1(X)Q_2(Y) 
$$
is divisible by an univariate polynomial which  contradicts 
Lemma~\ref{lem:uni_factors}. 
So, this means that $R(X)$ can be written in the following form,
$$
  R(X) = \(\frac{X}{X-s_1}\)^{m},
$$
and this concludes the proof.
\end{proof}

Notice that the condition that $f(X)/g(X)$ is not a perfect power of a
polynomial is necessary, indeed if $f(X) = (h(X))^n$ and $g(X)=1$
with $f(X), h(X) \in  \bFp[X]$ then
$f(X) - \lambda^n f(Y)$ is divisible by $h(X)-\lambda h(Y)$
for any $\lambda \in \bFp$. 

\subsection{Integral points on affine curves}

We need the following estimate
of Bombieri and Pila~\cite{BP} on
the number of integral points on polynomial curves.

\begin{lem}\label{lem:BombPila}
Let $\cC$ be a plane absolutely irreducible curve of degree $n\ge 2$
and let $H\ge \exp(n^6)$. Then the number of integral points on $\cC$ 
inside of the square $[0,H]\times [0,H]$ is at most
$H^{1/n}\exp(12\sqrt{n\log H\log \log H})$.
\end{lem}

\subsection{Small values of linear functions}

We need a result about small values of residues modulo $p$
of several linear functions. Such a result has been 
derived in~\cite[Lemma~3.2]{CSZ} from 
 the Dirichlet pigeon-hole principle.
Here use a slightly more precise and explicit 
form of this result which is derived in~\cite{GG} 
from the {\it Minkowski theorem\/}. 

First we recall some standard notions of the 
theory of geometric lattices.
 
Let ${\vec{b}}_1,\ldots,{\vec{b}}_r$ 
be  $r$ linearly 
independent vectors in ${\R}^s$. The set
$$
\cL=\{\vec{z}  \ : \ \vec{z}=c_1\vec{b}_1+\ldots+
c_r\vec{b}_r,\quad c_1, \ldots, c_r\in\Z\}
$$
is called an {\it $r$-dimensional  lattice in $\R^s$\/} with
a {\it basis\/} $\{ {\vec{b}}_1,\ldots, {\vec{b}}_r\}$. 

To each lattice $\cL$ one can naturally associate its {\it
volume\/}
$$
\vol{\cL}  = \(\det \(B^tB\)\)^{1/2},
$$
where $B$ is the $s\times r$ matrix whose columns are 
formed by  the  vectors 
${\vec{b}}_1,\ldots,{\vec{b}}_r$ and
 $B^t$ is the 
transposition of $B$. It is well known that $\vol{\cL}$ 
 does not depend on the choice of  the   basis
$\{{\vec{b}}_1,\ldots,{\vec{b}}_r\}$, we refer to~\cite{GrLoSch}
for a background on lattices.

For a vector $\vec{u}$, let 
$$
\|\vec{u}\|_\infty = \max\{|u_1|, \ldots, |u_s|\}
$$
 denote its  {\it
infinity norm\/} of $\vec{u}= (u_1, \ldots, u_s) \in \R^s$. 

The famous {\it Minkowski theorem\/}, 
see~\cite[Theorem~5.3.6]{GrLoSch},  gives an upper bound
on the size 
of the shortest nonzero vector in any $r$-dimensional lattice
$\cL$ in terms of its volume.

\begin{lem}
\label{lem:MinkB} For any $r$-dimensional lattice $\cL$
we have
$$
 \min \left\{ \|\vec{z}\|_\infty   \colon \ \vec{z}  \in
\cL\setminus \{\vec{0}\}\right\} \le  \(\vol{\cL}\)^{1/r}. 
$$
\end{lem}

For an integer $a$ we use $\flp{a}$ to denote the smallest by
absolute value residue of $a$ modulo $p$, that is
$$
\flp{a} = \min_{k\in \Z} |a - kp|.
$$
The following result is essentially contained in~\cite[Theorem~2]{GG}.
We include here a short proof.

\begin{lem}
\label{lem:Red}
For any real numbers $V_1,\ldots, V_{s} $ with
$$
p> V_1,\ldots, V_{s} \ge 1 \mand   V_1\ldots V_{s} > p^{s-1}
$$
and  
integers $b_1, \ldots, b_{s}$,  there exists an integer $v$ with $\gcd(v,p) =1$
such that
$$
\flp{b_i v} \le V_i, \qquad i =1, \ldots, s.
$$
\end{lem}

\begin{proof}
Without loss of the generality, we can take $b_1=1$. 
We introduce the following notation,
\begin{equation}
\label{eq:euclid}
V = \prod_{i=1}^{s}V_i
\end{equation}
and consider the lattice $\cL$ generated by the columns 
of the following  matrix
$$
B =  \left (
    \begin{matrix}
      b_sV/V_s & 0  &\ldots &0 & pV/V_s  \\
      b_{s-1}V/V_{s-1} & 0  &\ldots & pV/V_{s-1} & 0 \\
      \vdots &  \vdots &\vdots &\vdots & \vdots\\
   b_2V/V_2  & pV/V_2 & \ldots & 0 & 0\\
     V/V_1 & 0 & \ldots &0 & 0\\
    \end{matrix}\right ).
$$
Clearly the volume of $\cL$ is 
$$
\vol{\cL} = \frac{V}{V_1} \prod_{j=2}^s \frac{pV}{V_j} = V^{s-1}p^{s-1} \le V^s
$$
by~\eqref{eq:euclid} and the conditions on the size of the product $V_1\ldots V_{s}$. 
Consider a nonzero vector with the minimum infinity norm inside $\cL$. By the
definition of $\cL$, this vector
is a linear combination of the columns of $B$ with integer
coefficients, that is, it can be written in the following way 
\begin{equation*}
\(\frac{c_1V}{V_1}, \frac{(c_1b_2+c_2p)V}{V_2},\ldots, \frac{(c_1b_s+c_sp)V}{V_s}\),\quad 
c_1,\ldots, c_s\in\Z.
\end{equation*}
By Lemma~\ref{lem:MinkB} and the bound
on the volume of $\cL$, the following inequality holds,
\begin{equation*}
  \max\left \{\left |\frac{c_1V}{V_1}\right |, \left |\frac{(c_1b_2+c_2p)V}{V_2}\right |,\ldots,
  \left |\frac{(c_1b_s+c_sp)V}{V_s}\right |\right \}\le V.
\end{equation*}
From here, it is trivial to check that if we choose $v = c_1$, then 
\begin{itemize}
\item $\flp{v}= \flp{c_1}  \le V_1$,
\item $\flp{v b_i} = \flp{c_1 b_i} \le V_i, 
  \qquad i =2, \ldots, s$,
\end{itemize}
which finishes the proof.
\end{proof}

\section{Main Results}

\begin{thm} 
\label{thm:N bound}
Let  $\psi(X) = f(X)/g(X)$ where $f,g\in \F_p[X]$ 
relatively prime  polynomials of degree $d$ and $e$ respectively with $d+e\ge 1$.
We define
$$
\ell = \min\{d,e\}, \qquad 
m = \max\{d,e\}
$$
and set
$$
k=(\ell+1)\(\ell m - \ell^2 + m^2 +  m\) \mand 
s = 2m\ell+2m- \ell^2.
$$
Assume that $\psi$ 
is   not   a perfect power of another rational function over $\bFp$. 
Then for any interval $\cI$  of   
$H$ consecutive
integers and a subgroup $\cG$ of $\F_p^*$ of order $T$, 
we have 
$$
N_\psi(\cI, \cG)\ll (1 + H^{\rho} p^{-\vartheta}) H^{\tau+o(1)} T^{1/2},
$$
 where 
$$
\vartheta = \frac{1}{2s}, \qquad  \rho = \frac{k}{2s}, 
\qquad \tau = \frac{1}{2(\ell+m)},
$$
and  the implied constant depends on $d$ and $e$.
\end{thm}

\begin{proof} Clearly we can assume that 
\begin{equation}
\label{eq:const c}
H \le c p^{2\vartheta/(2\rho-1)}
\end{equation}
for some constant $c>0$ which may depend on $d$ and $e$ 
as otherwise one easily verifies that $H^{\rho} p^{-\vartheta}\ge 1$
and
$$
H^{\rho+\tau}p^{-\vartheta} \ge H^{1/2},
$$
and hence
the desired bound 
is weaker than the trivial estimate 
$$
N_\psi(\cI, \cG)\ll \min\{H,T\} \le H^{1/2} T^{1/2}.
$$
 
Making the transformation $X \mapsto X+u$, 
we can assume that  $\cI = \{1, \ldots, H\}$.
Let $1 \le x_1 < \ldots < x_r \le H$ be all $r=N_\psi(\cI, \cG)$ values of $x \in \cI$
with  $\psi(x) \in \cG$.

Let $\Lambda$ be the set of exceptional values of $\lambda \in \bFp$ 
described in Lemma~\ref{lem:LinFact}. 
We see that there are only at most $4m^3r$ pairs $(x_i,x_j)$, $1 \le
i,j \le r$, 
for which  $\psi(x_i)/\psi(x_j)\in \Lambda$.
 Indeed, if $x_j$ is fixed, then
$\psi(x_i)$ can take  
at most $4m^2$ values of the form $\lambda \psi(x_j)$, with $\lambda \in \Lambda$,

Furthermore, each value
$\lambda \psi(x_j)$ can be taken by $\psi(x_i)$  for at most $D$
possible values of $i=1, \ldots, r$. 

We now assume that $r > 8m^3$ as otherwise there is nothing to prove. 
Therefore, 
there is $\lambda \in \cG \setminus \Lambda$  such that 
\begin{equation}
\label{eq:cong f}
\psi(x) \equiv \lambda \psi(y) \pmod p
\end{equation}
for at least 
\begin{equation}
\label{eq:rT}
\frac{r^2 -4m^3r}{T} \ge \frac{r^2}{2T} 
\end{equation}
pairs $(x,y)$ with $x,y \in \{1, \ldots, H\}$.

Let 
$$
f(X)g(Y) - \lambda f(Y) g(X) = \sum_{i=0}^m \sum_{j=0}^m b_{i,j} X^iY^j
$$
Let 
$$
\cH =\{(i,j)~:~i,j=0, \ldots, m, \ i+j \ge 1, \min\{i,j\} \le \ell \}.
$$
Clearly the noncostant terms $b_{i,j} X^iY^j$ of $f(X)g(Y) - \lambda f(Y) g(X)$
are supported only on the subscripts $(i,j) \in \cH$. 
We have
$$
\#\cH = 2(m+1)(\ell+1) - (\ell+1)^2 -1 = s 
$$
We now apply Lemma~\ref{lem:Red} with $s = \# \cH$ and
the vector  
$\(b_{i,j}\)_{(i,j) \in \cH}$.

We also define the quantities $U$ and 
$V_{i,j}$, $(i,j) \in \cH$ 
by the relations
$$
V_{i,j}H^{i+j} = U, \qquad (i,j) \in \cH, 
$$
thus
$$
\prod_{(i,j) \in \cH} V_{i,j} 
=  2p^{s-1}.
$$
By  Lemma~\ref{lem:Red} there is an integer $v$ with $\gcd(v,p)=1$ 
such that 
$$
\flp{b_{i,j}v} \le V_{i,j}  
$$
for every $(i,j) \in \cH$.

We have
\begin{equation*}
\begin{split}
\sum_{(i,j) \in \cH} & (i +j)
=2 \sum_{i=0}^m \sum_{j=0}^\ell (i+j) - 
\sum_{i=0}^\ell \sum_{j=0}^\ell (i+j) \\
&= 2 \sum_{i=0}^m \((\ell+1)i+\frac{\ell(\ell+1)}{2}\) - 
\sum_{i=0}^\ell\((\ell+1)i+\frac{\ell(\ell+1)}{2}\)\\
&= 2\(\frac{(\ell+1)m(m+1)}{2} + \frac{\ell(\ell+1)(m+1)}{2}\) \\
& \qquad \qquad\qquad\qquad - 
\frac{\ell(\ell+1)^2}{2}- \frac{\ell(\ell+1)^2}{2} = k.\end{split}
\end{equation*}

Certainly it is easy to evaluate $V_{i,j}$ and $V^{(\lambda)}_{i,j}$, 
$(i,j) \in \cH$ explicitly, however it is enough for us to note that 
we have
$$
U^{s} H^{k}
= 2p^{s-1}.
$$

Hence
\begin{equation}
\label{eq:U}
U = 2^{1/3}p^{1-1/s} H^{k/s}.
\end{equation}
We also assume that the constant $c$  in~\eqref{eq:const c}
is small enough so the condition 
$$
\max_{(i,j) \in \cH} \left\{V_{i,j}, V^{(\lambda)}_{i,j}\right\} = UH^{-1}
 <p $$
is satisfied. 

Let $F(X,Y) \in \Z[X]$  and $G(X,Y) \in \Z[X]$ be polynomials 
with coefficients in the interval $[-p/2,p/2]$, 
obtained by reducing $vf(X)g(Y)$ and
$v\lambda f(Y)g(X)$ modulo $p$, respectively. 
Clearly~\eqref{eq:cong f} implies 
\begin{equation}
\label{eq:Cong FG}
F(x,y) \equiv  G(x,y) \pmod p.
\end{equation}
Furthermore,  since for $x,y\in \{1, \ldots, H\}$,
we see from~\eqref{eq:U} and the trivial estimate  
on the constant coefficients (that is, $|F(0)|, |G(0)| \le p/2$)
that 
$$
|F(x,y)-G(x,y)| \ll U+p \ll  p^{1-1/s} H^{k/s} + p, 
$$ 
which together with~\eqref{eq:Cong FG} implies that 
\begin{equation}
\label{eq:Eq FGz}
F(x,y) =  G(x,y) + zp
\end{equation}
for some integer $z \ll p^{-1/s} H^{k/s}+1$.

Clearly, for any integer $z$ the reducibility of  $F(X,Y) - G(X,Y)-pz$  over $\C$ implies the
reducibility of $F(X,Y) - G(X,Y)$  over $\bFp$, or equaivalently
$f(X)g(Y) -\lambda f(Y)g(X)$ over $\bFp$, which is 
impossible because $\lambda \not \in \Lambda$. 
 
Because  $F(X,Y) - G(X,Y)-pz\in \C[X,Y]$  is  irreducible over $\C$ and has 
degree $d$, we derive from
Lemma~\ref{lem:BombPila} that for every $z$ the equation~\eqref{eq:Eq FGz}
has at most $H^{1/(d+e)+o(1)}$ solutions.
Thus the congruence~\eqref{eq:cong f} has at most
$O\(H^{1/(d+e)+o(1)} \(p^{-1/s} H^{k/s}+1\)\)$ solutions.
This, together with~\eqref{eq:rT}, yields the inequality 
$$
\frac{r^2}{2T}  \ll H^{1/(d+e)+o(1)} \(p^{-1/s} H^{k/s}+1\),
$$ 
and concludes the proof.
\end{proof}

Clearly, in the case when $e=0$, that is, $\psi=f$ is a 
polynomial of degree $d\ge 2$, the bound of Theorem~\ref{thm:N bound}
takes form
$$
N_\psi(\cI, \cG)\ll \(1 + H^{(d+1)/4} p^{-1/4d}\) H^{1/2d+o(1)} T^{1/2}.
$$

\section{Comments}

Clearly Theorem~\ref{thm:N bound} also provides a bound
for the case where rational function $\psi = \varphi^s$, with $\varphi\in\bFp(X)$. 
This comes from the fact that
\begin{equation*}
  \psi(x)\in\cG \implies \varphi(x)\in\cG_0,
\end{equation*}
where $\cG_0$ is a multiplicative subgroup of $\bFp$ of order
bounded by $sT$. However the resulting bound depends now on the degrees 
of the polynomials associated with $\varphi$ rather than that of $\psi$.  

Another consequence  from Theorem~\ref{thm:N bound} is the following:
given an interval $\cI$ and a subgroup $\cG \in \F_p^*$,  satisfying $N_\psi(\cI, \cG)= \#\cI$ 
then 
$$
\#\cG \gg \min\{(\# \cI)^{2 - 2 \tau+o(1)}, 
(\# \cI)^{1-2\rho - 2\tau+o(1)}p^{2\vartheta} \}
$$ 
where the implied constant depends only on $d$ and $e$. However, we believe that this bound 
is very unlikely to be tight.

\section*{Acknowledgements}

D.~G-P. would like  to thank  Macquarie University for the
 support and hospitality during his  stay in Australia. 
 
During the preparation of this paper  D.~G-P. was supported by the 
Ministerio de Economia y Competitividad project TIN2011-27479-C04-04
and I.~S.  by the  Australian Research Council
Grants~DP130100237 and~DP140100118.

\end{document}